\documentclass[a4paper,10.5pt]{article}
\usepackage{amsmath,amsthm,amssymb,latexsym,graphicx}
\theoremstyle{definition}
\newtheorem{dfn}{Definition}[section]
\newtheorem{thm}[dfn]{Theorem}

\newtheorem{prop}[dfn]{Proposition}

\newtheorem{cor}[dfn]{Corollary}

\allowdisplaybreaks[4]

\makeatletter

\@addtoreset{equation}{section}
\makeatother

\title{Cotton tensor and conformal deformations of three-dimensional Ricci flow}
\author{Yoshihiro Umehara}
\date{}

\begin{document}

\maketitle

\begin{abstract}
In this paper, we study the deformation of the three-dimensional conformal structures by the Ricci flow. We drive the evolution equation of the Cotton-York tensor and the $L^1$-norm of it under the Ricci flow. In particular, we investigate the behavior of the $L^1$-norm of the Cotton-York tensor under the Ricci flow on three-dimensional simply-connected Riemannian homogeneous spaces which admit compact quotients. For a non-homogeneous case, we also investigate the behavior of the $L^1$-norm for the product metric of the Rosenau solution for the Ricci flow on $S^2$ and the standard metric of $S^1$. 
\end{abstract}

\section{Introduction}

Let $M^n$ be a $C^\infty$ manifold. A one-parameter family of Riemannian metrics $g(t)$ is called the {\it Ricci flow} if it satisfies
\begin{equation*}
\frac{\partial}{\partial t}g=-2\mathrm{Ric}_g.
\end{equation*}
We are interested in the properties of the Ricci flow from the viewpoint of three-dimensional conformal geometry. More precisely, we study the deformation of the three-dimensional conformal structures by the Ricci flow.

It is well known that the conformal flatness in dimension $n\ge4$ is equivalent to the vanishing of the Weyl tensor. In dimension $n=3$, the Weyl tensor vanishes identically, and hence the conformal flatness cannot be detected by the Weyl tensor. However, there is a conformally invariant tensor which in $n=3$ plays a role analogous to that of the Weyl tensor in $n\ge4$. This tensor is called the {\it Cotton tensor} and defined by   
\begin{equation*}
C_3=C_{ijk}:=\nabla_iR_{jk}-\nabla_jR_{ik}-\frac{1}{4}\left(\nabla_iRg_{jk}-\nabla_jRg_{ik}\right),
\end{equation*}
where $R_{ij}$ is the Ricci tensor, $R$ is the scalar curvature and $\nabla$ is the Levi-Civita connection. It can be shown that $C_3$ is conformally invariant and the conformal flatness is equivalent to $C_3=0$. By a direct computation, we can see that the following properties hold:
\begin{enumerate}
\item$C_{ijk}+C_{jik}=0$,
\item$C_{ijk}+C_{jki}+C_{kij}=0$,
\item$g^{ij}C_{ijk}=g^{ik}C_{ijk}=g^{jk}C_{ijk}=0$.
\end{enumerate}
We can write the Cotton tensor in an algebraically equivalent form which is called the {\it Cotton-York tensor} \cite{Y}
\begin{equation*}
C_2=C_{ij}:=g_{ik}\varepsilon^{klm}\left(\nabla_lR_{mj}-\frac{1}{4}\nabla_lRg_{mj}\right)=\frac{1}{2}g_{ik}\varepsilon^{klm}C_{lmj},
\end{equation*}
where $\varepsilon^{ijk}$ is a tensor constructed by using the completely anti-symmetric tensor density $\eta^{klm}$ of weight $+1$ with $\eta^{123}=1$ and the determinant of the metric tensor $g$ for the given coordinate system:
\begin{equation*}
\varepsilon^{klm}:=\frac{\eta^{klm}}{\sqrt{\det g}}.
\end{equation*}
The tensor $\varepsilon$ satisfies the following:
\begin{enumerate}
\item$\varepsilon_{ijk}\varepsilon^{ilm}=\delta_j^l\delta_k^m-\delta_j^m\delta_k^l$,
\item$\varepsilon_{ijk}\varepsilon^{ijl}=2\delta_k^l$,
\item$\varepsilon_{ijk}\varepsilon^{ijk}=6$,
\item$\nabla_i\varepsilon^{jkl}=\nabla_i\varepsilon_{jkl}=0$.
\end{enumerate}
From the relation $C_{ijk}=\varepsilon_{ijl}g^{lm}C_{mk}$, we can see that the conformal flatness is equivalent to $C_2=0$. The $(2,0)$-tensor $C_2$ has the following properties:
\begin{enumerate}
\item$C_{ij}=C_{ji}$  (symmetric),
\item$g^{ij}C_{ij}=0$  (trace-free),
\item$\nabla^iC_{ij}=0$  (divergence-free/transverse),
\item$|C_3|_g=\sqrt2|C_2|_g$, 
\end{enumerate}
where $|C_3|^2_g(x)=(g^{ip}g^{jq}g^{kr}C_{ijk}C_{pqr})(x)$ and $|C_2|^2_g(x)=({g^{ip}g^{jq}C_{ij}C_{pq}})(x)$.
We consider the $L^1$-norm of the Cotton-York tensor on a closed Riemannian manifold $(M^3,g)$
\begin{equation*}
C(g):=\int_{M^3}|C_2|_gd\mu_g, 
\end{equation*}
where $d\mu_g$ is the volume element of $g$. Note that the $L^1$-norm of the Cotton tensor differs from that of the Cotton-York tensor by $\sqrt 2$ multiple. By the property of the Cotton-York tensor, it is easy to see the conformal invariance of the $L^1$-norm and the equivalence between the vanishing of the $L^1$-norm and the conformal flatness. If the manifold is non-compact, we consider the $L^1$-norm on an arbitrary compact set $K$

\begin{equation*}
C_K(g):=\int_K|C_2|_gd\mu_g.
\end{equation*}
We are interested in the behavior of the $L^1$-norm under the Ricci flow.

As a fundamental result of the Ricci flow by R.~Hamilton \cite{H}, it is known that the Ricci flow starting at an initial metric with positive Ricci curvature on a three-dimensional closed manifold converges to a constant curvature metric up to scaling. Since a constant curvature metric is conformally flat, we can regard this result as a convergence to a conformally flat metric and a vanishing of the $L^1$-norm of the Cotton-York tensor. In general, the Ricci flow develops singularities, but it is not clear whether or not the conformal structure degenerates in the sense that the $L^1$-norm of the Cotton-York tensor blows up. This observations have motivated us to look into the properties of the Cotton-York tensor under the Ricci flow.

C.~Mantegazza, S.~Mongodi, and M.~Rimoldi \cite{Ma} described the evolution of the Cotton tensor $C_3$ under the Ricci flow. By using this evolution equation, we derive the evolution equation of the Cotton-York tensor $C_2$ and the $L^1$-norm of $C_2$ (\textbf{Proposition~\ref{cy}} and \textbf{Theorem~\ref{L}}). In particular, we investigate the behavior of the $L^1$-norm of the Cotton-York tensor on two separate contexts. The first is evolution of the $L^1$-norm under the Ricci flow on simply-connected three-dimensional Riemannian homogeneous spaces $M=G/H$ which admit compact quotients. Here $G$ is a transitive group of diffeomorphisms of $M$ and $H$ is the compact isotropy subgroup. We assume that $G$ is minimal, i.e.\ no proper subgroup of $G$ acts transitively on $M$. The second context is for the product metric of the Rosenau solution \cite{R} for the Ricci flow on $S^2$ (which is ancient and shrinks to a round point as $t\nearrow 0$) and the standard metric of $S^1$. Note that the product metric on $S^2\times S^1$ is also a solution to the Ricci flow. Recall the complete list of the Riemannian homogeneous spaces $M$ (\cite{M}, \cite{S}): $\mathbb{R}^3$, $\mathrm{SU}(2)$, $\widetilde{\mathrm{Isom}(\mathbb{R}^2)}$, $\widetilde{\mathrm{SL}(2,\mathbb{R})}$, the Heisenberg group, $\mathrm{Isom}(\mathbb{R}^1_1)$, $\mathbb{H}^3$, $\mathbb{H}^2\times \mathbb{R}$, and $S^2\times\mathbb{R}$, where the tildes denote the universal covering space, $\mathbb{R}^1_1$ is the two-dimensional Minkowski space, $\mathbb{H}^3$ is three-dimensional hyperbolic space, and $\mathbb{H}^2$ is two-dimensional hyperbolic space. Since homogeneous geometries on $\mathbb{H}^3$, $\mathbb{H}^2\times \mathbb{R}$, and $S^2\times\mathbb{R}$ are conformally flat, the $L^1$-norm of Cotton-York tensor for these geometries trivial. In the first six homogeneous spaces, for an arbitrary left invariant metric $g_0$, J.~Milnor \cite{M} provided a left invariant frame field $\{F_i\}_i^3$ (called the {\it Milnor frame} for $g_0$) such that
\begin{equation*}
g_0=A_0\,\omega^1\otimes\omega^1+B_0\,\omega^2\otimes\omega^2+C_0\,\omega^3\otimes\omega^3
\end{equation*}
where $A_0$, $B_0$, $C_0$ are positive constants and
\begin{equation*}
[F_2,F_3]=2\lambda F_1,\quad [F_3,F_1]=2\mu F_2,\quad [F_1,F_2]=2\nu F_3,
\end{equation*} 
where $\lambda,\mu,\nu\in \{-1,0,1\}$ and $\lambda\le\mu\le\nu$ are satisfied. Recall that the value of the triplet $\lambda$, $\mu$, $\nu$ completely determines the corresponding Lie group for the six homogeneous spaces. With respect to the Milnor frame, not only $g_0$ but also $\mathrm{Ric}_{g_0}$ are diagonalized. As $g_0$ and $\mathrm{Ric}_{g_0}$ remain diagonalized under the Ricci flow, it follows that the metric $g(t)$ evolves as 
\begin{equation*}
g(t)=A(t)\,\omega^1\otimes\omega^1+B(t)\,\omega^2\otimes\omega^2+C(t)\,\omega^3\otimes\omega^3
\end{equation*}
and that the Ricci flow equation becomes a system of three ODE's \cite{I} for $A(t)$, $B(t)$, and $C(t)$. J.~Isenberg and M.~Jackson \cite{I} studied the behavior of the normalized Ricci flow on all the homogeneous spaces. The behavior of the (unnormalized) Ricci flow on those spaces was studied by D.~Knopf and K.~McLeod \cite{K} (see also \cite{CK}). The Ricci flow on $\mathbb{R}^3$ is trivial. It becomes asymptotically round as $t\nearrow T<\infty$ on $\mathrm{SU}(2)$. It converges to the flat space as $t\nearrow\infty$ on $\widetilde{\mathrm{Isom}(\mathbb{R}^2)}$. For the other Lie groups, the each solution to the Ricci flow approaches a flat degenerate geometry of either two or one dimensions as $t\nearrow\infty$. We follow the calculations as done in these previous works.

We suppose that in the case of $\mathrm{SU}(2)$ and $\widetilde{\mathrm{SL}(2,\mathbb{R})}$, the initial metric $g_0$ satisfies $B_0=C_0$. The results for the six homogeneous spaces are summarized in the Table~1 (\textbf{Theorem~\ref{SU}}, \textbf{\ref{R2}}, \textbf{\ref{SL}}, \textbf{\ref{N}}, \textbf{\ref{R11}}). The main conclusions are the following:
\begin{enumerate}
\item In all cases, the each $L^1$-norm of the Cotton-York tensor converges to zero.  
\item If the initial metric $g_0$ on $\mathrm{SU}(2)$ satisfies $B_0=C_0$ and $A_0/B_0<1/2$, the $L^1$-norm has a unique local extremum at $t_0$ with $A(t_0)/B(t_0)=1/2$. 
\item In other cases, the each $L^1$-norm is strictly decreasing or identically zero.
\end{enumerate}
\begin{center}
\begin{tabular}{|l|l|}\hline
Lie group&Behavior of the $L^1$-norm $C(g)$, $C_K(g)$\\\hline
$\mathrm{SU}(2)$&$C(g)\to 0$ and it has a unique local extremum if $A_0/B_0<1/2$.\\
&$C(g)\searrow 0$ if $1/2\le A_0/B_0<1$ or $1<A_0/B_0$.\\
&$C(g)=0$ if $A_0=B_0$.\\
$\widetilde{\mathrm{Isom}(\mathbb{R}^2)}$&$C_K(g)\searrow 0$ if $A_0\ne B_0$.\\
&$C_K(g)=0$ if $A_0=B_0$.\\
$\widetilde{\mathrm{SL}(2,\mathbb{R})}$&$C_K(g)\searrow 0$.\\
Heisenberg&$C_K(g)\searrow 0$.\\
$\mathrm{Isom}(\mathbb{R}^1_1)$&$C_K(g)\searrow 0$.\\
$\mathbb{R}^3$&$C_K(g)=0$.\\\hline
\end{tabular}
\end{center}
\begin{center}
Table 1: The behavior of the $L^1$-norm of the Cotton-York tensor
\end{center}
The $L^1$-norm of the Cotton-York tensor for the product metric of the Rosenau solution and the standard metric of $S^1$ is strictly decreasing and converges to zero as $t\nearrow 0$ (\textbf{Theorem~\ref{Pro}}).

It is interesting that in these examples the $L^1$-norm of the solution to the Ricci flow starting at the initial metric with non-positive scalar curvature is strictly decreasing. The following are topics for further investigation:
\begin{itemize}
\item The monotonicity of the $L^1$-norm of the Cotton-York tensor (or the lack of it).
\item The characterization of the Riemannian manifold at the time which the $L^1$-norm takes a local extremum.
\end{itemize}
 
\noindent Acknowledgments.\, The author would like to thank Sumio Yamada and Shigetoshi Bando for their useful comments and encouragement.

\section{The evolution equation of the $L^1$-norm of the Cotton-York tensor}
For any tensor $T$, $S$ such as $T_{ij}$, $S_{ij}$, we define $\langle T,S\rangle_g:=g^{ip}g^{jq}T_{ij}S_{pq}$, $T^2:=T_{ik}g^{kl}T_{lj}$, $\mathrm{div}_gT:=\nabla^iT_{ij}$, and $\Delta_gT:=g^{ij}\nabla_i\nabla_jT_{ij}$. Our goal in this section is to derive the following evolution equation.
\begin{thm}\label{L}
Let $(M^3,g(t)), 0\le t<T$ be a solution of the Ricci flow on a closed manifold. Suppose the norm of the Cotton-York tensor $C_2$ does not vanish in $M\times[0,T)$. Then the $L^1$-norm of $C_2$ satisfies the following evolution equation
\begin{align*}
&\frac{d}{dt}\int_M|C_2|_{g}d\mu_g\\
&=\int_M\frac{1}{|C_2|_{g}}(\Delta_g|C_2|_g^2-2|\nabla C_2|_g^2-16\langle\mathrm{Ric}, C_2^2\rangle_g+7R|C_2|^2_{g}\\
&-4\langle\mathrm{Ric}, \mathrm{div}_gD\rangle_g+4\langle\mathrm{Ric}^2, \mathrm{div}_gC_3\rangle_g-2\langle\nabla R, \mathrm{div}_g(\mathrm{div}_gC_3)\rangle_g)d\mu_g,
\end{align*}
where $D=D_{ijk}:=C_{ijp}g^{pq}R_{qk}$. 
\end{thm}
The evolution equation of the Cotton tensor $C_3$ under the Ricci flow is obtained by Mantegazza, Mongodi, and Rimoldi.
\begin{prop}\label{Cotton}(\cite{Ma})
Let $(M^3,g(t))$ be a solution of the Ricci flow. Then the Cotton tensor $C_3$ satisfies the following evolution equation  
\begin{equation*}
\begin{split}
\frac{\partial}{\partial t}C_{ijk}&=\Delta_g C_{ijk}+g^{pq}R_{pj}(C_{kqi}+C_{kiq})+5g^{pq}R_{kp}C_{jiq}+g^{pq}R_{pi}(C_{qkj}+C_{jkq})\\
&+2RC_{ijk}+2g^{pq}g^{rs}R_{pr}C_{sjq}g_{ki}-2g^{pq}g^{rs}R_{pr}C_{siq}g_{kj}\\
&+\frac{1}{2}\left(\nabla_i|\mathrm{Ric}|_g^2\right)g_{kj}-\frac{1}{2}\left(\nabla_j|\mathrm{Ric}|_g^2\right)g_{ki}+\frac{R}{2}(\nabla_jR)g_{ki}-\frac{R}{2}(\nabla_iR)g_{kj}\\
&+2g^{pq}R_{pi}\nabla_jR_{qk}-2g^{pq}R_{pj}\nabla_iR_{qk}+R_{kj}\nabla_iR-R_{ki}\nabla_jR.
\end{split}
\end{equation*}
\end{prop}
By using Proposition \ref{Cotton}, we obtain the evolution equation of the Cotton-York tensor $C_2$ under the Ricci flow.    
\begin{prop}\label{cy}
Let $(M^3,g(t))$ be a solution of the Ricci flow. Then the Cotton-York tensor $C_2$ satisfies the following evolution equation
\begin{align*}
\frac{\partial}{\partial t}C_{ij}&=\Delta_g C_{ij}-5g^{pq}R_{ip}C_{qj}-5g^{pq}C_{iq}R_{pj}+2\langle C_2,\mathrm{Ric}\rangle_gg_{ij}+4RC_{ij}\\
&+\frac{1}{2}g_{ik}g_{jm}\varepsilon^{klm}\nabla_l|\mathrm{Ric}|_g^2+\frac{R}{2}g_{ik}g_{jl}\varepsilon^{klm}\nabla_mR+2g_{ik}g^{pq}\varepsilon^{klm}R_{pl}\nabla_mR_{qj}\\
&+g_{ik}\varepsilon^{klm}R_{jm}\nabla_lR.
\end{align*}
\end{prop}
\begin{cor}\label{cor}
Let $(M^3,g(t))$ be a solution of the Ricci flow. Then the squared norm of the Cotton-York tensor $C_2$ satisfies the following evolution equation
\begin{align*}
\frac{\partial}{\partial t}|C_2|^2_g&=\Delta_g|C_2|^2_g-2|\nabla C_2|^2_g-16\langle\mathrm{Ric}, C^2_2\rangle_g+8R|C_2|_g^2\\
&-4\langle\mathrm{Ric},\mathrm{div}_gD\rangle_g+4\langle\mathrm{Ric}^2,\mathrm{div}_gC_3\rangle_g-2\langle\nabla R,\mathrm{div}_g(\mathrm{div}_gC_3)\rangle_g.
\end{align*}
\end{cor}
\begin{proof}[Proof of Proposition \ref{cy}]
Note that 
\begin{equation*}
\frac{\partial}{\partial t}\varepsilon^{klm}=R\varepsilon^{klm}.
\end{equation*}
Indeed,
\begin{align*}
\frac{\partial}{\partial t}\varepsilon^{klm}&=\frac{\partial}{\partial t}\left(\frac{\eta^{klm}}{\sqrt{\det (g_{ij})}}\right)=-\frac{\eta^{klm}}{\det (g_{ij})}\cdot\left(-R\sqrt{\det (g_{ij})}\right)=R\varepsilon^{klm}.
\end{align*}
By this equation and Proposition \ref{Cotton},
\begin{align*}
\frac{\partial}{\partial t}C_{ij}&=\frac{1}{2}\left(\frac{\partial}{\partial t}g_{ik}\right)\varepsilon^{klm}C_{lmj}+\frac{1}{2}g_{ik}\left(\frac{\partial}{\partial t}\varepsilon^{klm}\right)C_{lmj}+\frac{1}{2}g_{ik}\varepsilon^{klm}\left(\frac{\partial}{\partial t}C_{lmj}\right)\\
&=-R_{ik}\varepsilon^{klm}C_{lmj}+\frac{1}{2}g_{ik}R\varepsilon^{klm}C_{lmj}+\frac{1}{2}g_{ik}\varepsilon^{klm}\\
&\times\{\Delta_g C_{lmj}+g^{pq}R_{pm}(C_{jql}+C_{jlq})+5g^{pq}R_{jp}C_{mlq}+g^{pq}R_{pl}(C_{qjm}+C_{mjq})\\
&+2RC_{lmj}+2g^{pq}g^{rs}R_{pr}C_{smq}g_{jl}-2g^{pq}g^{rs}R_{pr}C_{slq}g_{jm}\\
&+\frac{1}{2}\left(\nabla_l|\mathrm{Ric}|_g^2\right)g_{jm}-\frac{1}{2}\left(\nabla_m|\mathrm{Ric}|_g^2\right)g_{jl}+\frac{R}{2}(\nabla_mR)g_{jl}-\frac{R}{2}(\nabla_lR)g_{jm}\\
&+2g^{pq}R_{pl}\nabla_mR_{qj}-2g^{pq}R_{pm}\nabla_lR_{qj}+R_{jm}\nabla_lR-R_{jl}\nabla_mR\}.
\end{align*}
We compute each term by using the identities $C_{ij}=\frac{1}{2}g_{ik}\varepsilon^{klm}C_{lmj}$, $C_{ijk}=\varepsilon_{ijl}g^{lm}C_{mk}$, and the properties of $C_3$, $C_2$, $\varepsilon$.  
\begin{align*}
&(\textrm{1st term of RHS})=-R_{ik}\varepsilon^{klm}\varepsilon_{lmp}g^{pq}C_{qj}=-R_{ik}\cdot2\delta_p^kg^{pq}C_{qj}=-2R_{ip}g^{pq}C_{qj}.\\
&(\textrm{2nd})=RC_{ij}.\\
&(\textrm{3rd})=\frac{1}{2}g_{ik}\varepsilon^{klm}\Delta_g C_{lmj}=\Delta_g\left(\frac{1}{2}g_{ik}\varepsilon^{klm}C_{lmj}\right)=\Delta_g C_{ij}.\\
&(\textrm{4th})=\frac{1}{2}g_{ik}\varepsilon^{klm}g^{pq}R_{pm}C_{jql}=\frac{1}{2}g_{ik}\varepsilon^{klm}g^{pq}R_{pm}(-C_{qlj}-C_{ljq})\\
&=-\frac{1}{2}g_{ik}\varepsilon^{klm}g^{pq}R_{pm}\varepsilon_{qlr}g^{rs}C_{sj}-\frac{1}{2}g_{ik}\varepsilon^{klm}g^{pq}R_{pm}\varepsilon_{ljr}g^{rs}C_{sq}\\
&=-\frac{1}{2}g_{ik}g^{pq}g^{rs}(\delta^k_q\delta^m_r-\delta^k_r\delta^m_q)R_{pm}C_{sj}+\frac{1}{2}g_{ik}g^{pq}g^{rs}(\delta^k_j\delta^m_r-\delta^k_r\delta^m_j)R_{pm}C_{sq}\\
&=-\frac{1}{2}g^{rs}R_{ir}C_{sj}+\frac{1}{2}RC_{ij}+\frac{1}{2}\langle C_2,\mathrm{Ric}\rangle_gg_{ij}-\frac{1}{2}g^{pq}C_{iq}R_{pj}.\\
&(\textrm{5th})=\frac{1}{2}g_{ik}\varepsilon^{klm}g^{pq}R_{pm}C_{jlq}=\frac{1}{2}g_{ik}\varepsilon^{klm}g^{pq}R_{pm}\varepsilon_{jlr}g^{rs}C_{sq}\\
&=\frac{1}{2}g_{ik}g^{pq}g^{rs}(\delta^k_j\delta^m_r-\delta^k_r\delta^m_j)R_{pm}C_{sq}=\frac{1}{2}g_{ij}g^{pq}g^{rs}R_{pr}C_{sq}-\frac{1}{2}g_{ir}g^{pq}g^{rs}R_{pj}C_{sq}\\
&=\frac{1}{2}\langle C_2,\mathrm{Ric}\rangle_gg_{ij}-\frac{1}{2}g^{pq}C_{iq}R_{pj}.\\
&(\textrm{6th})=\frac{5}{2}g_{ik}\varepsilon^{klm}g^{pq}R_{jp}\varepsilon_{mlr}g^{rs}C_{sq}=-\frac{5}{2}g_{ik}g^{pq}g^{rs}\cdot2\delta^k_rR_{jp}C_{sq}=-5g^{pq}C_{iq}R_{pj}.\\
&(\textrm{7th})=\frac{1}{2}g_{ik}\varepsilon^{klm}g^{pq}R_{pl}C_{qjm}=-\frac{1}{2}g_{ik}\varepsilon^{klm}g^{pq}R_{pl}C_{jqm}=(\textrm{4th}).\\
&(\textrm{8th})=\frac{1}{2}g_{ik}\varepsilon^{klm}g^{pq}R_{pl}C_{mjq}=-\frac{1}{2}g_{ik}\varepsilon^{klm}g^{pq}R_{pl}C_{jmq}=(\textrm{5th}).\\
&(\textrm{9th})=2RC_{ij}.\\
&(\textrm{10th})=g_{ik}\varepsilon^{klm}g^{pq}g^{rs}R_{pr}\varepsilon_{sma}g^{ab}C_{bq}g_{jl}=-g_{ik}g^{pq}g^{rs}g^{ab}g_{jl}(\delta^k_s\delta^l_a-\delta^k_a\delta^l_s)R_{pr}C_{bq}\\
&=-g^{pq}R_{pi}C_{jq}+g^{pq}C_{iq}R_{pj}.\\
&(\textrm{11th})=\frac{1}{2}g_{ik}\varepsilon^{kml}\cdot(-2)g^{pq}g^{rs}R_{pr}C_{smq}g_{jl}=\frac{1}{2}g_{ik}\varepsilon^{klm}\cdot2g^{pq}g^{rs}R_{pr}C_{smq}g_{jl}=(\textrm{10th}).\\
&(\textrm{12th})=\frac{1}{4}g_{ik}g_{jm}\varepsilon^{klm}\nabla_l|\mathrm{Ric}|_g^2.\\
&(\textrm{13th})=\frac{1}{2}g_{ik}\varepsilon^{klm}\cdot\left(-\frac{1}{2}\nabla_m|\mathrm{Ric}|_g^2\right)g_{jl}=\frac{1}{2}g_{ik}\varepsilon^{kml}\cdot\frac{1}{2}\left(\nabla_m|\mathrm{Ric}|_g^2\right)g_{jl}=(\textrm{12th}).\\
&(\textrm{14th})=\frac{R}{4}g_{ik}g_{jl}\varepsilon^{klm}\nabla_mR.\\
&(\textrm{15th})=\frac{1}{2}g_{ik}\varepsilon^{klm}\cdot\left(-\frac{R}{2}\nabla_lR\right)g_{jm}=\frac{1}{2}g_{ik}\varepsilon^{kml}\cdot\frac{R}{2}(\nabla_lR)g_{jm}=(\textrm{14th}).\\
&(\textrm{16th})=g_{ik}g^{pq}\varepsilon^{klm}R_{pl}\nabla_mR_{qj}.\\
&(\textrm{17th})=\frac{1}{2}g_{ik}\varepsilon^{klm}\cdot(-2)g^{pq}R_{pm}\nabla_lR_{qj}=\frac{1}{2}g_{ik}\varepsilon^{kml}\cdot2g^{pq}R_{pm}\nabla_lR_{qj}=(\textrm{16th}).\\
&(\textrm{18th})=\frac{1}{2}g_{ik}\varepsilon^{klm}R_{jm}\nabla_lR.\\
&(\textrm{19th})=\frac{1}{2}g_{ik}\varepsilon^{klm}\cdot(-1)R_{jl}\nabla_mR=\frac{1}{2}g_{ik}\varepsilon^{kml}R_{jl}\nabla_mR=(\textrm{18th}).
\end{align*}
Hence, we obtain the result.
\end{proof}

\begin{proof}[Proof of Corollary \ref{cor}]
By Proposition \ref{cy}, 
\begin{align*}
&\frac{\partial}{\partial t}|C_2|^2_g\\
&=2\left(\frac{\partial}{\partial t}g^{i_1i_2}\right)g^{j_1j_2}C_{i_1j_1}C_{i_2j_2}+2g^{i_1i_2}g^{j_1j_2}\left(\frac{\partial}{\partial t}C_{i_1j_1}\right)C_{i_2j_2}\\
&=4R^{i_1i_2}C_{i_1i_2}^2+2g^{i_1i_2}g^{j_1j_2}\\
&\times\{\Delta_g C_{i_1j_1}-5g^{pq}R_{i_1p}C_{qj_1}-5g^{pq}C_{i_1q}R_{pj_1}+2\langle C_2,\mathrm{Ric}\rangle_gg_{i_1j_1}+4RC_{i_1j_1}\\
&+\frac{1}{2}g_{i_1k}g_{j_1m}\varepsilon^{klm}\nabla_l|\mathrm{Ric}|_g^2+\frac{R}{2}g_{i_1k}g_{j_1l}\varepsilon^{klm}\nabla_mR+2g_{i_1k}g^{pq}\varepsilon^{klm}R_{pl}\nabla_mR_{qj_1}\\
&+g_{i_1k}\varepsilon^{klm}R_{j_1m}\nabla_lR\}\times C_{i_2j_2}.
\end{align*}
We compute each term by using $\frac{1}{2}g_{ik}\varepsilon^{klm}C_{lmj}$, $C_{ijk}=\varepsilon_{ijl}g^{lm}C_{mk}$, and the properties of $C_3$, $C_2$, $\varepsilon$. 
\begin{align*}
&(\textrm{1st term of RHS})=4\langle\mathrm{Ric}, C^2_2\rangle_g.\\
&(\textrm{2nd})=2g^{i_1i_2}g^{j_1j_2}(\Delta_g C_{i_1j_1})C_{i_2j_2}=2\langle\Delta C_2, C_2\rangle_g=\Delta_g|C_2|^2_g-2|\nabla C_2|^2_g.\\
&(\textrm{3rd})=-10g^{i_1i_2}g^{pq}R_{i_1p}C^2_{i_2q}=-10\langle\mathrm{Ric}, C_2^2\rangle_g.\\
&(\textrm{4th})=-10g^{j_1j_2}g^{pq}C^2_{qj_2}R_{pj_1}=-10\langle C^2_2,\mathrm{Ric}\rangle_g.\\
&(\textrm{5th})=4\langle C_2,\mathrm{Ric}\rangle_gg^{i_2j_2}C_{i_2j_2}=4\langle C_2,\textrm{Ric}\rangle_g\mathrm{tr}_gC_2=0.\\
&(\textrm{6th})=8RC_{i_1j_1}g^{i_1i_2}g^{j_1j_2}C_{i_2j_2}=8R|C_2|_g^2.\\
&(\textrm{7th})=\delta_k^{i_2}\delta^{j_2}_m\varepsilon^{klm}\nabla_l|\mathrm{Ric}|_g^2C_{i_2j_2}=\varepsilon^{klm}(\nabla_l|\mathrm{Ric}|_g^2)C_{km}=0.\\
&(\textrm{8th})=\delta_k^{i_2}\delta^{j_2}_lR\varepsilon^{klm}\nabla_mRC_{i_2j_2}=R\varepsilon^{klm}(\nabla_mR)C_{kl}=0.\\
&(\textrm{9th})=4g^{i_1i_2}g^{j_1j_2}g_{i_1k}g^{pq}\varepsilon^{klm}R_{pl}(\nabla_mR_{qj_1})\cdot\frac{1}{2}g_{i_2a}\varepsilon^{ars}C_{rsj_2}\\
&=2g^{pq}\varepsilon^{i_2lm}\varepsilon_{i_2rs}R_{pl}(\nabla_mR_{qj_1})C^{rsj_1}=2g^{pq}(\delta^l_r\delta^m_s-\delta^l_s\delta^m_r)R_{pl}(\nabla_mR_{qj_1})C^{rsj_1}\\
&=-4g^{pq}R_{pr}C^{srj_1}\nabla_sR_{j_1q}=-4g^{pq}R_{pr}\nabla_s(C^{srj_1}R_{j_1q})+4g^{pq}R_{pr}(\nabla_sC^{srj_1})R_{j_1q}\\
&=-4\langle\mathrm{Ric},\mathrm{div}_gD\rangle_g+4\langle\mathrm{Ric}^2,\mathrm{div}_gC_3\rangle_g.\\
&(\textrm{10th})=2g^{i_1i_2}g^{j_1j_2}g_{i_1k}\varepsilon^{klm}R_{j_1m}(\nabla_lR)\cdot\frac{1}{2}g_{i_2a}\varepsilon^{ars}C_{rsj_2}\\
&=\varepsilon^{i_2lm}\varepsilon_{i_2rs}R_{j_1m}(\nabla_lR)C^{rsj_1}=(\delta^l_r\delta^m_s-\delta^l_s\delta^m_r)R_{j_1m}(\nabla_lR)C^{rsj_1}\\
&=2(\nabla^rR)C_{rsj_1}R^{sj_1}=2(\nabla_rR)(-\nabla^q\nabla^pC_{pqr})\\
&=-2\langle\nabla R,\mathrm{div}_g(\mathrm{div}_gC_3)\rangle_g,
\end{align*}
where we use the identity $\nabla^j\nabla^iC_{ijk}=-C_{klm}R^{lm}$ (see for example \cite[p.\ 9]{Ch}). Hence, we obtain the result.
\end{proof}
Theorem \ref{L} follows from Corollary \ref{cor} and $\frac{d}{dt}d\mu_g=-R\mathit{d\mu_g}$.

\section{Examples of behavior of the $L^1$-norm}
\subsection{The Lie group $\mathrm{SU}(2)$}
We consider the Ricci flow $g(t)$ starting at a left invariant metric $g_0$ on $\mathrm{SU}(2)$, and fix a Milnor frame for $g_0$ such that $\lambda=\mu=\nu=-1$. Note that $\mathrm{SU}(2)$ is identified topologically with standard three-sphere of radius one embedded in $\mathbb{R}^4$.

The Ricci tensor of $g$ is
\begin{equation*}
\begin{split}
R(F_1,F_1)&=4-2\frac{B^2+C^2-A^2}{BC},\\
R(F_2,F_2)&=4-2\frac{C^2+A^2-B^2}{CA},\\
R(F_3,F_3)&=4-2\frac{B^2+A^2-C^2}{BA}.
\end{split}
\end{equation*}
Then the Ricci flow equation is equivalent to the system of ODE's  
\begin{equation*}
\left\{
\begin{split}
&\frac{d}{dt}A=-8+4\frac{B^2+C^2-A^2}{BC},\\
&\frac{d}{dt}B=-8+4\frac{C^2+A^2-B^2}{CA},\\
&\frac{d}{dt}C=-8+4\frac{B^2+A^2-C^2}{BA}.
\end{split}
\right.
\end{equation*}
\begin{prop}(\cite[Proposition 1.17]{CK})
For any choice of initial data $A_0$, $B_0$, $C_0>0$, the unique solution $g(t)$ exists for a maximal finite time interval $0\le t<T<\infty$. The metric $g(t)$ becomes asymptotically round as $t\nearrow T$. 
\end{prop}
Now we are interested in the behavior of the $L^1$-norm of the Cotton-York tensor $C_2$. Since the $L^1$-norm is very complicated for general initial data, we assume that $B_0=C_0$. Then $B(t)=C(t)$ holds from the symmetry in the Ricci flow equation.
\begin{thm}\label{SU}
For any choice of initial data $A_0$, $B_0=C_0>0$, the behavior of the $L^1$-norm $C(g)$ of the Cotton-York is the following: 
\begin{enumerate}
\item If $0<A_0/B_0<1/2$, $C(g)$ has a unique local extremum at $t_0$ with $A(t_0)/B(t_0)=1/2$ and converges to zero as $t\to T$.
\item If $1/2\le A_0/B_0<1$ or $1<A_0/B_0$, $C(g)$ is strictly decreasing and converges to zero as $t\to T$.
\item If $A_0=B_0$, $C(g)$ is identically zero.
\end{enumerate}
\end{thm}
\begin{proof}
In this case, the Ricci flow equation is reduced to
\begin{equation*}
\frac{d}{dt}A=-4\left(\frac{A}{B}\right)^2,\quad\frac{d}{dt}B=-8+4\frac{A}{B},
\end{equation*}
and the scalar curvature is 
\begin{equation*}
R=\frac{2(4B-A)}{B^2}.
\end{equation*}
Note that $A_0/B_0=1$, $A_0/B_0<1$, and $A_0/B_0>1$ are preserved under the Ricci flow, and $\lim_{t\nearrow T}A=\lim_{t\nearrow T}B=0$ in all cases.

The Cotton-York tensor is 
\begin{equation*}
C_2(F_1,F_1)=8\frac{A^\frac{3}{2}}{B^2}\left(\frac{A}{B}-1\right),\quad C_2(F_2,F_2)=C_2\left(F_3,F_3\right)=4\frac{A^\frac{1}{2}}{B}\left(1-\frac{A}{B}\right).
\end{equation*}
Then for an arbitrary compact set $K$,
\begin{equation*}
\int_{K}|C_2(t)|_{g(t)}d\mu_{g(t)}=
4\sqrt6\frac{A}{B}\left|\frac{A}{B}-1\right|\mathrm{Vol}(K,g_{S^3}),
\end{equation*}
where $g_{S^3}$ is the standard metric of radius one on $S^3$. In particular,
\begin{equation*}
\int_{S^3}|C_2(t)|_{g(t)}\mathit{d\mu_{g(t)}}=
\begin{cases}
4\sqrt6A/B(1-A/B)\mathrm{Vol}(S^3,g_{S^3}), & 0<A_0/B_0\le1,\\
4\sqrt6A/B(A/B-1)\mathrm{Vol}(S^3,g_{S^3}), & 1\le A_0/B_0. 
\end{cases}
\end{equation*}
If $A_0=B_0$, $C(g)$ is identically zero. We assume that $A_0\ne B_0$. We show that as $t\nearrow T$, $A/B\nearrow 1$ if $A_0/B_0<1$ and $A/B\searrow 1$ if $A_0/B_0>1$. Indeed, 
\begin{equation*}
\frac{d}{dt}\frac{A}{B}=\frac{-4\left(A/B\right)^2B-A\left(-8+4(A/B)\right)}{B^2}
=8\frac{A}{B^2}\left(1-\frac{A}{B}\right),
\end{equation*}
hence $A/B$ is strictly increasing if $A_0/B_0<1$ and strictly decreasing if $A_0/B_0>1$. Since $A/B$ is bounded and monotone, it converges to some constant $\alpha>0$. By l'H\^{o}spital' rule, 
\begin{equation*}
\alpha=\lim_{t\nearrow T}\frac{A}{B}=\lim_{t\nearrow T}\frac{-4\left(A/B\right)^2}{-8+4\left(A/B\right)}=\frac{-4\alpha^2}{-8+4\alpha}.
\end{equation*}
Hence we obtain $\alpha=1$. 

We define the functions $f$ and $h$ on $\mathbb{R}$ respectively as
\begin{equation*}
f(x):=x(1-x)\quad\mathrm{and}\quad h(x):=x(x-1).
\end{equation*}
Since $A/B\nearrow 1$ if $A_0/B_0<1$, the function $f(A/B)$ has a maximal value at $t_0$ with $A(t_0)/B(t_0)=1/2$ and $f(A/B)\to 0$ if $A_0/B_0<1/2$, and $f(A/B)\searrow 0$ if $1/2\le A_0/B_0<1$. Since $A/B\searrow 1$ if $1<A_0/B_0$, the function $h(A/B)\searrow 0$ if $1<A_0/B_0$. Hence if $0<A_0/B_0<1/2$, the $L^1$-norm $C(g)$ has a maximal value at $t_0$ with $A(t_0)/B(t_0)=1/2$ and converges to zero as $t \to T$. If $1/2\le A_0/B_0<1$ or $1<A_0/B_0$, it is strictly decreasing and converges to zero as $t \to T$. 
\end{proof}

\subsection{The Lie group $\widetilde{\mathrm{Isom}(\mathbb{R}^2)}$}
We consider the Ricci flow $g(t)$ starting at a left invariant metric $g_0$ on $\widetilde{\mathrm{Isom}(\mathbb{R}^2)}$, and fix a Milnor frame for $g_0$ such that $\lambda=\mu=-1$ and $\nu=0$.

The Ricci tensor $g$ is
\begin{equation*}
R(F_1,F_1)=-2\frac{B^2-A^2}{BC},\;
R(F_2,F_2)=-2\frac{A^2-B^2}{AC},\;
R(F_3,F_3)=-2\frac{(A-B)^2}{AB},
\end{equation*}
and the scalar curvature $g$ is
\begin{equation*}
R=-2\frac{(A-B)^2}{ABC}.
\end{equation*}
Then the Ricci flow equation is equivalent to the system of ODE's
\begin{equation*}
\left\{
\begin{split}
&\frac{d}{dt}A=4\frac{B^2-A^2}{BC},\\
&\frac{d}{dt}B=4\frac{A^2-B^2}{AC},\\
&\frac{d}{dt}C=4\frac{(A-B)^2}{AB}.
\end{split}
\right.
\end{equation*}
By the direct computation, we can show $\frac{d}{dt}(AB)=\frac{d}{dt}(C(A+B))=0$. 
\begin{prop}\label{PR2}(\cite{K})
For any choice of initial data $A_0$, $B_0$, $C_0>0$, the unique solution $g(t)$ exists for all positive time. For any $\varepsilon>0$, there exists $T_\varepsilon>0$ such that
\begin{equation*}
\left|A-\sqrt{A_0B_0}\right|\le\varepsilon,\quad\left|B-\sqrt{A_0B_0}\right|\le\varepsilon,\quad\left|C-\frac{C_0}{2}\left(\sqrt{\frac{A_0}{B_0}}+\sqrt{\frac{B_0}{A_0}}\right)\right|\le\varepsilon
\end{equation*}
for all $t\ge T_\varepsilon$. 
Moreover, as $t \nearrow \infty$, $B/A\nearrow 1$ if $B_0/A_0<1$, $B/A\searrow 1$ if $1<B_0/A_0$, and $B/A=1$ if $B_0/A_0=1$.
\end{prop}
The behavior of the $L^1$-norm of the Cotton-York tensor $C_2$ is given by the next result.
\begin{thm}\label{R2}
For any choice of initial data $A_0$, $B_0$, $C_0>0$, the behavior of the $L^1$-norm $C_K(g)$ of the Cotton-York tensor on an arbitrary compact set $K$ is the following: 
\begin{enumerate}
\item If $A_0\ne B_0$, $C_K(g)$ is strictly decreasing and converges to zero as $t\to \infty$.
\item If $A_0=B_0$, $C_K(g)$ is identically zero.
\end{enumerate}
\end{thm}
\begin{proof}
The Cotton-York tensor is
\begin{equation*}
\begin{split}
C_2(F_1,F_1)&=\frac{4A}{(ABC)^\frac{3}{2}}(2A^3-B^3-A^2B),\\
C_2(F_2,F_2)&=\frac{4B}{(ABC)^\frac{3}{2}}(2B^3-A^3-AB^2),\\
C_2(F_3,F_3)&=-\frac{4C}{(ABC)^\frac{3}{2}}(A+B)(A-B)^2.\\
\end{split}
\end{equation*}
Then for an arbitrary compact set $K$,
\begin{equation*}
\begin{split}
&\int_K|C_2|_gd\mu_g\\
&=\left(6\left(\frac{A}{B}\right)^3-6\left(\frac{A}{B}\right)^2+2\left(\frac{A}{B}\right)+6\left(\frac{B}{A}\right)^3-6\left(\frac{B}{A}\right)^2+2\left(\frac{B}{A}\right)-4\right)^{\frac{1}{2}}\\
&\times\frac{4(A_0B_0)^\frac{1}{2}}{C}\mathrm{Vol}(K,h)\\
\end{split}
\end{equation*}
where $h=\omega^1\otimes\omega^1+\omega^2\otimes\omega^2+\omega^3\otimes\omega^3$.

If $A_0=B_0$, $C_K(g)$ is identically zero. We assume that $A_0\ne B_0$. We define the function $f$ on $\mathbb{R}$ as
\begin{equation*}
f(x):=\left(6\left(\frac{1}{x}\right)^3-6\left(\frac{1}{x}\right)^2+2\left(\frac{1}{x}\right)+6x^3-6x^2+2x-4\right)^{\frac{1}{2}}.
\end{equation*}
The function $f$ is strictly decreasing if $0<x\le 1$ and strictly increasing if $1<x$. By Proposition \ref{PR2}, as $t \nearrow \infty$, $f(B/A)\searrow 0$ if $B_0/A_0<1$ and $f(B/A)\searrow 0$ if $1<B_0/A_0$. Clearly $1/C$ is strictly decreasing, hence $C_K(g)$ is strictly decreasing and converges to zero as $t \to \infty$.  
\end{proof}

\subsection{The Lie group $\widetilde{\mathrm{SL}(2,\mathbb{R})}$}

We consider the Ricci flow $g(t)$ starting at a left invariant metric $g_0$ on $\widetilde{\mathrm{SL}(2,\mathbb{R})}$, and fix a Milnor frame such that $\lambda=-1$ and $\mu=\nu=1$. 

The Ricci tensor of $g$ is 
\begin{equation*}
\begin{split}
R(F_1,F_1)&=-2\frac{(B-C)^2-A^2}{BC},\\
R(F_2,F_2)&=-2\frac{(A+C)^2-B^2}{AC},\\
R(F_3,F_3)&=-2\frac{(A+B)^2-C^2}{AB}.
\end{split}
\end{equation*}
Then the Ricci flow equation is equivalent to the system of ODE's
\begin{equation*}
\left\{
\begin{split}
&\frac{d}{dt}A=4\frac{(B-C)^2-A^2}{BC},\\
&\frac{d}{dt}B=4\frac{(A+C)^2-B^2}{AC},\\
&\frac{d}{dt}C=4\frac{(A+B)^2-C^2}{AB}.
\end{split}
\right.
\end{equation*}
\begin{prop}(\cite{K})
For any choice of initial data $A_0$, $B_0$, $C_0>0$, the unique solution $g(t)$ exists for all positive time. There exists $A_\infty=A_\infty(A_0,B_0,C_0)>0$ such that for any $\varepsilon>0$, there exists $T_\varepsilon>0$ such that
\begin{equation*}
\left|A-A_\infty\right|\le\varepsilon,\quad\left|\frac{d}{dt}B-8\right|\le\varepsilon,\quad\left|\frac{d}{dt}C-8\right|\le\varepsilon
\end{equation*}
for all $t\ge T_\varepsilon$.  
\end{prop}
Now we are interested in the behavior of the $L^1$-norm of the Cotton-York tensor $C_2$. Since the $L^1$-norm is very complicated for general initial data, we assume that $B_0=C_0$. Then $B(t)=C(t)$ holds from the symmetry in the Ricci flow equation. 
\begin{thm}\label{SL}
For any choice of initial data $A_0$, $B_0=C_0>0$, the $L^1$-norm $C_K(g)$ of the Cotton-York tensor on an arbitrary compact set $K$ is strictly decreasing and converges to zero as $t\to\infty$. 
\end{thm}
\begin{proof}
In this case, the Ricci flow equation is reduced to
\begin{equation*}
\frac{d}{dt}A=-4\left(\frac{A}{B}\right)^2,\quad\frac{d}{dt}B=4\frac{A}{B}+8,
\end{equation*}
and the scalar curvature is 
\begin{equation*}
R=-\frac{2(A+4B)}{B^2}.
\end{equation*}
The Cotton-York tensor is
\begin{equation*}
C_2(F_1,F_1)=\frac{8A^3(A+B)}{(AB^2)^\frac{3}{2}},\quad
C_2(F_2,F_2)=C_2(F_3,F_3)=-\frac{4A^2B(A+B)}{(AB^2)^\frac{3}{2}}.
\end{equation*}
Then for an arbitrary compact set $K$,
\begin{equation*}
\int_K|C_2|_gd\mu_g=4\sqrt{6}\frac{A}{B}\left(1+\frac{A}{B}\right)\mathrm{Vol}(K,h),
\end{equation*}
where $h=\omega^1\otimes\omega^1+\omega^2\otimes\omega^2+\omega^3\otimes\omega^3$. 

The function $A/B$ is strictly decreasing and converges to zero as $t\to\infty$. Indeed,
\begin{equation*}
\begin{split}
\frac{d}{dt}\frac{A}{B}&=\frac{-4(A/B)^2B-A\{4(A/B)+8\}}{B^2}=-8\frac{A}{B^2}\left(\frac{A}{B}+1\right)<0,
\end{split}
\end{equation*}
and $\lim_{t\to\infty}(A/B)=A_\infty/\infty=0$. Hence $C_K(g)$ is strictly decreasing and converges to zero as $t\to\infty$.
\end{proof}

\subsection{The Heisenberg group}
We consider the Ricci flow $g(t)$ starting at a left invariant metric $g_0$ on the Heisenberg group, and fix a Milnor frame for $g_0$ such that $\lambda=-1$ and $\mu=\nu=0$. 

The Ricci tensor of $g$ is
\begin{equation*}
R(F_1,F_1)=2\frac{A^2}{BC},\quad R(F_2,F_2)=-2\frac{A}{C},\quad R(F_3,F_3)=-2\frac{A}{B},
\end{equation*}
and the scalar curvature of $g$ is
\begin{equation*}
R=-2\frac{A}{BC}.
\end{equation*}
Then the Ricci flow equation is equivalent to the system of ODE's
\begin{equation*}
\left\{
\begin{split}
&\frac{d}{dt}A=-4\frac{A^2}{BC},\\
&\frac{d}{dt}B=4\frac{A}{C},\\
&\frac{d}{dt}C=4\frac{A}{B}.
\end{split}
\right.
\end{equation*}
\begin{prop}(\cite{K})
For any choice of initial data $A_0$, $B_0$, $C_0>0$, the unique solution $g(t)$ exists for all positive time. Moreover, the above system of ODE's is solved explicitly:
\begin{equation*}
\begin{split}
&A=A_0^{\frac{2}{3}}B_0^{\frac{1}{3}}C_0^{\frac{1}{3}}\left(12t+\frac{B_0C_0}{A_0}\right)^{-\frac{1}{3}},\\
&B=A_0^{\frac{1}{3}}B_0^{\frac{2}{3}}C_0^{-\frac{1}{3}}\left(12t+\frac{B_0C_0}{A_0}\right)^{\frac{1}{3}},\\
&C=A_0^{\frac{1}{3}}B_0^{-\frac{1}{3}}C_0^{\frac{2}{3}}\left(12t+\frac{B_0C_0}{A_0}\right)^{\frac{1}{3}}
\end{split}
\end{equation*}
for $t\in (-B_0C_0/A_0,\infty)$.
\end{prop}
The behavior of the $L^1$-norm of the Cotton-York tensor $C_2$ is given by the following:
\begin{thm}\label{N}
For any choice of initial data $A_0$, $B_0$, $C_0>0$, the $L^1$-norm $C_K(g)$ of the Cotton-York tensor on an arbitrary compact set $K$ is strictly decreasing and converges to zero as $t\to\infty$. 
\end{thm}
\begin{proof}
The Cotton-York tensor is
\begin{equation*}
C_2(F_1,F_1)=\frac{8A^2}{BC}\sqrt{\frac{A}{BC}},\;
C_2(F_2,F_2)=-\frac{4A^2}{C\sqrt{ABC}},\;
C_2(F_3,F_3)=-\frac{4A^2}{B\sqrt{ABC}}.
\end{equation*}
Then for an arbitrary compact set $K$,
\begin{equation*}
\begin{split}
\int_{K}|C_2(t)|_{g(t)}d\mu_{g(t)}&=2\sqrt{6}\frac{A^2}{BC}\mathrm{Vol}(K,h)\\
&=2\sqrt{6}A_0^\frac{2}{3}B_0^\frac{1}{3}C_0^\frac{1}{3}\left(12t+\frac{B_0C_0}{A_0}\right)^{-\frac{4}{3}}\mathrm{Vol}(K,h),
\end{split}
\end{equation*}
where $h=\omega^1\otimes\omega^1+\omega^2\otimes\omega^2+\omega^3\otimes\omega^3$. 

Hence $C_K(g)$ is strictly decreasing and converges to zero as $t\to\infty$.
\end{proof}

{\subsection{The Lie group $\mathrm{Isom}(\mathbb{R}^1_1)$}
We consider the Ricci flow $g(t)$ starting at a left invariant metric $g_0$ on $\mathrm{Isom}(\mathbb{R}^1_1)$, and fix a Milnor frame for $g_0$ such that $\lambda=-1$, $\mu=0$, and $\nu=1$.

The Ricci tensor $g$ is
\begin{equation*}
R(F_1,F_1)=-2\frac{C^2-A^2}{BC},\;
R(F_2,F_2)=-2\frac{(A+C)^2}{AC},\;
R(F_3,F_3)=-2\frac{A^2-C^2}{AB}
\end{equation*}
and the scalar curvature $g$ is
\begin{equation*}
R=-2\frac{(A+C)^2}{ABC}.
\end{equation*}
Then the Ricci flow equation is equivalent to the system of ODE's
\begin{equation*}
\left\{
\begin{split}
&\frac{d}{dt}A=4\frac{C^2-A^2}{BC},\\
&\frac{d}{dt}B=4\frac{(A+C)^2}{AC},\\
&\frac{d}{dt}C=4\frac{A^2-C^2}{AB}.
\end{split}
\right.
\end{equation*}
By the direct computation, we can show $\frac{d}{dt}(AC)=\frac{d}{dt}(B(C-A))=0$.
\begin{prop}\label{PR11}(\cite{K})\label{PR11}
For any choice of initial data $A_0$, $B_0$, $C_0>0$, the unique solution $g(t)$ exists for all positive time. For any $\varepsilon>0$, there exists $T_\varepsilon>0$ such that
\begin{equation*}
\left|A-\sqrt{A_0C_0}\right|\le\varepsilon,\quad\left|C-\sqrt{A_0C_0}\right|\le\varepsilon,\quad\left|\frac{d}{dt}B-16\right|\le\varepsilon
\end{equation*}
for all $t\ge T_\varepsilon$. 
Moreover, as $t \nearrow \infty$, $A/C\nearrow 1$ if $A_0/C_0<1$, $A/C\searrow 1$ if $1<A_0/C_0$, and $A/C=1$ if $A_0/C_0=1$.
\end{prop}
Now we are interested in the behavior of the $L^1$-norm of the Cotton-York tensor $C_2$. 
\begin{thm}\label{R11}
For any choice of initial data $A_0$, $B_0$, $C_0>0$, the behavior of the $L^1$-norm $C_K\left(g\right)$ of the Cotton-York tensor on an arbitrary compact set $K$ is strictly decreasing and converges to zero as $t\to \infty$.
\end{thm}
\begin{proof}
The Cotton-York tensor is
\begin{equation*}
\begin{split}
C_2(F_1,F_1)&=\frac{4A(A+C)}{B\sqrt{ABC}}\left(2\frac{A}{C}+\frac{C}{A}-1\right),\\
C_2(F_2,F_2)&=\frac{4(A+C)}{\sqrt{ABC}}\left(\frac{C}{A}-\frac{A}{C}\right),\\
C_2(F_3,F_3)&=-\frac{4C(A+C)}{B\sqrt{ABC}}\left(2\frac{C}{A}+\frac{A}{C}-1\right).\\
\end{split}
\end{equation*}
Then for an arbitrary compact set $K$,
\begin{equation*}
\begin{split}
\int_K|C_2|_gd\mu_g=\frac{4(A+C)}{B}\left(6\frac{A}{C}\left(\frac{A}{C}-1\right)+6\frac{C}{A}\left(\frac{C}{A}-1\right)+8\right)^\frac{1}{2}\mathrm{Vol}(K,h),
\end{split}
\end{equation*}
where $h=\omega^1\otimes\omega^1+\omega^2\otimes\omega^2+\omega^3\otimes\omega^3$. 

We show that $(A+C)/B$ is strictly decreasing and converges to zero as $t\to\infty$. Indeed,
\begin{equation*}
\begin{split}
&\frac{d}{dt}\frac{A+C}{B}\\
&=\frac{\{4(C^2-A^2)/(BC)-4(A^2-C^2)/(AB)\}B-(A+C)\{(A+C)^2/(AC)\}}{B^2}\\
&=-\frac{8(A^3+A^2C+AC^2+C^3)}{AB^2C}<0,
\end{split}
\end{equation*}
and $\lim_{t\to\infty}(A+C)/B=2\sqrt{A_0C_0}/\infty=0$. 

If $A_0=C_0$, the $L^1$-norm $C_K(g)$ is reduced to 
\begin{equation*}
\int_K|C_2|_gd\mu_g=\frac{8\sqrt 2(A+C)}{B}\mathrm{Vol}(K,h).
\end{equation*}
Hence $C_K(g)$ is strictly decreasing and converges to zero as $t\to\infty$. We assume that $A_0\ne C_0$. We define the function $f$ on $\mathbb{R}$ as
\begin{equation*}
f(x):=\left(6x(x-1)+6\frac{1}{x}\left(\frac{1}{x}-1\right)+8\right)^\frac{1}{2}.
\end{equation*}
The function $f$ is strictly decreasing if $0<x\le 1$ and strictly increasing if $1<x$. By Proposition \ref{PR11}, as $t\nearrow\infty$, $f(A/C)\searrow 2\sqrt{2}$ if $A_0/C_0<1$ and $f(A/C)\searrow 2\sqrt{2}$ if $1<A_0/C_0$. Hence $C_K(g)$ is strictly decreasing and converges to zero as $t\to \infty$.  
\end{proof}

\subsection{The product metric of the Rosenau solution and the standard metric of $S^1$ }

Let $(\mathbb{R}\times S^1(2),dx^2+d\theta^2)$ denote the flat cylinder, where $\theta\in S^1(2)=\mathbb{R}/4\pi\mathbb{Z}$. We define a solution $g(x,\theta,t)$ for $t<0$ to the Ricci flow on $\mathbb{R}\times S^1(2)$ by
\begin{equation*}
g(x,\theta,t)=u(x,\theta,t)(dx^2+d\theta^2)=\frac{\sinh(-t)}{\cosh x+\cosh t}(dx^2+d\theta^2).
\end{equation*}
It is known that the solution $g(x,\theta,t)$ extends to the complete ancient solution to the Ricci flow on $S^2$ (see \cite[pp.\ 162-164]{C}, \cite[pp.\ 31-34]{CK}). This solution on $S^2$ is called the {\it Rosenau solution}. We denote this extended solution by $g$ as well. The scalar curvature of $g$ on $\mathbb{R}\times S^1(2)$ is
\begin{equation*}
R(x,\theta,t)=\frac{\cosh t\cdot\cosh x+1}{\sinh(-t)(\cosh x+\cosh t)}>0
\end{equation*} 
and the scalar curvature $R(\pm\infty,t)$ at the poles $x=\pm\infty$ is
\begin{equation*}
R(\pm\infty,t)=\lim_{|x|\to\infty}\frac{\cosh t\cdot\cosh x+1}{\sinh(-t)(\cosh x+\cosh t)}=\coth(-t)>0.
\end{equation*}
Moreover, the curvature $R(\pm\infty,t)$ at the poles is the maximum curvature of $\left(S^2,g(t)\right)$ for all $t<0$, since we have 
\begin{equation*}
\frac{\partial}{\partial x}R=\frac{\sinh x\cdot\sinh(-t)}{(\cosh x+\cosh t)^2}>0
\end{equation*}
for all $x>0$.
Since $\lim_{t\nearrow 0}R(\pm\infty,t)=\infty$, the Rosenau solution is ancient but not eternal. Due to the fact that for all $(x,\theta)\in\mathbb{R}\times S^1(2)$
\begin{equation*}
\lim_{t\nearrow 0}\frac{R(x,\theta,t)}{R(\pm\infty,t)}=\lim_{t\nearrow 0}\frac{\cosh t\cdot\cosh x+1}{\cosh t\left(\cosh x+\cosh t\right)}=1,
\end{equation*}
the solution shrinks to a round point.

Using the Rousenau solution, we define the Ricci flow on $S^2\times S^1$ by $h(t)=g(t)+d\varphi^2$ for $t<0$, where $d\varphi^2$ is the standard metric of radius one on $S^1$. 
\begin{thm}\label{Pro}
The $L^1$-norm $C(h)$ of the Cotton-York tensor $C_2$ for the product metric $h$ of the Rosenau solution for the Ricci flow on $S^2$ and the standard metric of $S^1$ is strictly decreasing and converges to zero as $t\to 0$.
\end{thm}
\begin{proof}
On the local coordinate $(x^1,x^2,x^3):=(x,\theta,\varphi)$, the Ricci tensor is 
\begin{equation*}
R_{11}=\frac{\cosh t\cdot\cosh x+1}{2(\cosh x+\cosh t)^2},\quad
R_{22}=\frac{\cosh t\cdot\cosh x+1}{2(\cosh x+\cosh t)^2},\quad
R_{33}=0,
\end{equation*}
and the scalar curvature is
\begin{equation*}
R=\frac{\cosh t\cdot\cosh x+1}{\sinh(-t)(\cosh x+\cosh t)}.
\end{equation*}
The Cotton-York tensor $C_2$ is 
\begin{equation*}
C_{23}=C_{32}=\frac{\sinh x\cdot\sinh(-t)}{4(\cosh x+\cosh t)^2}.
\end{equation*}
Then $L^1$-norm is given by the following:
\begin{align*}
\int_{S^2\times S^1}|C_2(t)|_{h(t)}d\mu_{h(t)}&=\int_{S^1}\left(\int_{S^2}|C_2(t)|_{h(t)}d\mu_{g(t)}\right)d\mu_{d\varphi^\textrm 2}\\
&=2\pi\int_{\mathbb{R}\times S^1(2)}|C_2(t)|_{h(t)}d\mu_{u(x,t)(\mathit{dx}^\textrm 2+d\theta^\textrm 2)}\\
&=2\pi\int_{S^1(2)}\left(\int_{\mathbb R}|C_2(t)|_{h(t)}u(x,t)d\mu_{dx^\textrm 2}\right)d\mu_{d\theta^\textrm 2}\\
&=8\pi^2\int_{\mathbb R}\frac{1}{2\sqrt{2}}\sqrt{\frac{\sinh^2x\cdot\sinh(-t)}{\left(\cosh x+\cosh t\right)^3}}\cdot\frac{\sinh(-t)}{\cosh x+\cosh t}dx\\
&=4\sqrt{2}\pi^2\int_0^\infty\sqrt{\frac{\sinh^2x\cdot\sinh(-t)}{\left(\cosh x+\cosh t\right)^3}}\cdot\frac{\sinh(-t)}{\cosh x+\cosh t}dx\\
&=4\sqrt{2}\pi^2\int_0^\infty\sqrt{\frac{\sinh^3(-t)}{\left(\cosh x+\cosh t\right)^5}}\cdot\sinh x\,dx\\
&=4\sqrt{2}\pi^2\int_1^\infty\sqrt{\frac{\sinh^3(-t)}{\left(y+\cosh t\right)^5}}dy\\
&=\frac{8\sqrt{2}\pi^2}{3}\left(\frac{\sinh(-t)}{1+\cosh(-t)}\right)^\frac{3}{2}.
\end{align*}
Hence $C\left(h\right)$ is strictly decreasing and converges to zero as $t\to 0$.
\end{proof}

\small{}

\textsl{Yoshihiro Umehara}

\textsl{Mathematical Institute, Tohoku University Sendai 980-8578, Japan}

sb1m06@math.tohoku.ac.jp
\end{document}